\documentclass{article}

\usepackage[mathlines]{lineno}

\usepackage[english]{babel}
\usepackage{graphicx}
\usepackage{textcomp}
\usepackage{amssymb}
\usepackage{amsmath}
\usepackage{setspace}
\usepackage{geometry}
\usepackage{bm}
\usepackage{centernot}

\newtheorem{remark}{Remark}
\newtheorem{definition}{Definition}
\newtheorem{theorem}{Theorem}
\newtheorem{lemma}{Lemma}
\newtheorem{example}{Example}
\newenvironment{proof}[1][Proof:]{\begin{trivlist}
\item[\hskip \labelsep {\bfseries #1}]}{\end{trivlist}}

\newcommand{\id}{{\mathbf 1}}
\newcommand{\supp}{\operatorname{supp}}

\newcommand{\Aut}{\mathrm{Aut}}

\usepackage{xcolor}

\begin{document}
\title{Homogeneous operators and homogeneous integral operators}

\author{Zhirayr Avetisyan \and Alexey Karapetyants}

\maketitle

\begin{abstract}
We introduce and study in a general setting the concept of homogeneity of an operator and, in particular, the notion of homogeneity of an integral operator. In the latter case, homogeneous kernels of such operators are also studied. The concept of homogeneity is associated with transformations of a measure - measure dilations, which are most natural in the context of  our general research scheme. For the study of integral operators, the notions of weak and strong homogeneity of the kernel are introduced. The weak case is proved to generate a homogeneous operator in the sense of our definition, while the stronger condition corresponds to the most relevant specific examples - classes of homogeneous integral operators on various metric spaces, and allows us to obtain an explicit general form for the kernels of such operators. The examples given in the article - various specific cases - illustrate general statements and results given in the paper and at the same time are of interest in their own way.
\end{abstract}



\section{Introduction}
In this paper we investigate in a general setting the concept of homogeneity of an operator and, in particular, the concept of homogeneity of an integral operator, and in the latter case, homogeneous kernels of such operators are also studied.

This study is inspired by the well-known theory of a special class of integral operators, namely, the class of operators with
homogeneous kernels in $\mathbb{R}^n$ with degree of homogeneity $-n$. For operators with homogeneous kernels we refer to the books
\cite{KarapetiantsSamko-book-rus, KarapetiantsSamko-book-eng} and the review paper \cite{KarapetiantsSamko-FCAA-1998}. We recall that the one-dimensional theory of operators with homogeneous kernels goes back to the Hardy-Littlewood-P\'olya theory; see \cite{HLP-Book}.

The Fredholm theory for operators with homogeneous kernels is well known; integral equations with such operators have been extensively studied in various settings. First of all, this concerns the framework of Lebesgue spaces, but now there are also results in the framework of Morrey and others function spaces, see e.g. \cite{Umarkhadzhiev, LukkassenPerssonSamko, NSamko-2020, AvsyankinEurasianMath, AvsyankinMathNotes}.

The definition of homogeneity for such integral operators, as customary, is related to the homogeneity property of their kernels. Namely, an operator whose kernel is homogeneous with respect to uniform scalings and invariant with respect to rotations is called an operator with homogeneous kernel.

However, the concept of homogeneity for an operator, not necessarily integral, is certainly a more general and broader issue. Even for an integral operator, the integral kernel does not directly interact with the metric but rather with the measure. The class of transformations which transform the measure in a way similar to what conformal transformations do is far larger, and it is therefore natural to study integral operators that are homogeneous with respect to such general, measure dilations.

At the same time our expectations are that techniques analogous to the classical ones developed in the preceding works on integral operators with homogeneous kernels will become available due to the appropriate behaviour of the measure. These expectations are supported by specific examples, which we also discuss in detail in Section \ref{Homogeneous integral kernels and homogeneous integral operators over some domains}.

So, basing the concept of homogeneous operators on measure dilations, given a measure $\mu$ on a measurable space $M$ we consider the measure dilation induced by a transformation $\varphi:M\to M$ such that $\mu(\varphi(A))=\lambda_\varphi\mu(A)$ for all measurable subsets $A\subset M$, where $\lambda_\varphi>0$ is a constant that depends only on $\varphi$.

Basically, what we do is the next step from the case in which only measure preserving transformations are considered. We allow a scalar factor $\lambda_\varphi$ in the definition given above, and hence operators invariant with respect to a group of such transformations are an extension of usual convolution-type operators, which are known to be originated by invariant measures.

When it comes to a particular, but important case - integral operators, a natural question arises. How does the concept of homogeneity of an operator correlate with the properties of homogeneity of the kernel? To clarify this question, we need to introduce the definitions of weak and strong homogeneity of a function (kernel). The kernel with weak homogeneity generates an integral operator which is homogeneous in the sense of our definition. At the same time, the most profound results are obtained for the case of operators with strong kernel homogeneity. In fact, after certain general considerations, we pass to integral operators with strongly homogeneous kernels, and it is for such kernels that we obtain results formulated in the most constructive form. 

We note that recently the study of special classes of operators, more precisely, integral operators with one or another property of invariance or symmetry, has acquired significant attention. In this regard, we would like to mention the studies of the class of so-called Hausdorff operators, see \cite{LL,LM1} and \cite{KarapetyantsLiflyand}. This class of operators differs essentially from operators with homogeneous kernels in a multidimensional situation, but the two classes are the same in the one-dimensional case. Also, we mention the paper on the class of Hausdorff - Berezin operators \cite{KS-CAOT-2019} and the paper on the class of Hadamard-Bergman operators \cite{KS-HB-2020, KA-HB-2021}. Operators in the latter class are actually integral realizations of multiplier operators considered on a space of holomorphic functions, and in fact, it is a specific example of a class of homogeneous integral operators in our interpretation. We discuss these operators below among other examples.

The paper is organized as follows. We prefer to go from the general to the particular, formulating the main ideas and definitions at the most reasonably general level, and further clarifying these concepts in cases where it is possible to be more specific.

Section \ref{Backgrounds, motivation and general thoughts} is an introductory part, which briefly provides information about classical operators with homogeneous kernels - the main motivation for this work. There we discuss also the concept of measure dilations in a general context. Finally, the notions of weak and strong homogeneity of the kernel of an integral operator are given.

Section \ref{Dilations, homogeneous operators and homogeneous integral operators} contains our main general statements on homogeneous operators and homogeneous integral operators. It is also the key point of study where the geometry of dilations is involved in our research. We separate and analyze two cases, which in the text below are designated as case A and case B. These cases differ in whether or not it is possible to reduce the situation under consideration to an invariant measure by means of certain unitary transformations. In other words, whether it is possible to reduce the operator in question to a convolution-type operator on a group. These two different situations are further explained with specific examples. Towards the end of the section we consider strongly homogeneous kernels. From this moment on, we drop the word "strongly" and continue to call them homogeneous kernels, since our next steps are associated with only such kernels. We obtain a general representation for such kernels, which we use later on in concrete examples.

Section \ref{Homogeneous integral kernels and homogeneous integral operators over some domains} contains our main particular results that are of independent interest and also serve as illustrations for theoretical aspects presented in Sections \ref{Backgrounds, motivation and general thoughts} and \ref{Dilations, homogeneous operators and homogeneous integral operators}. In this section, we consider concrete situations in which it is possible to calculate explicitly the corresponding transformations of the measure and to present with maximum detail the general form of the homogeneous kernel of an integral operator.

We begin our consideration with the case of the cylinder $\mathbb {R} \times \mathbb {T} $, and then move on to the study of homogeneous integral kernels on $ \mathbb{R}^2 $. Here we note the formula (\ref{GeneralSolutionEuclid}) which gives a precise general form of the homogeneous integral kernel. This clarifies the hypothesis that was previously accepted on the general form of an integral kernel, see Remark \ref{Rem:HomogeneousOperatorsProblemEuclid}. As a special example of a class of integral operators with homogeneous kernels on $ \mathbb{R}^ 2 $, or more precisely, on the unit disc, we then proceed with the representation of Hadamard-Bergman convolution operators in the form of integral operators with homogeneous kernels.

Next, we move on to weight measures on the unit disc and consider two different situations; the Haar measure corresponding to the the Poincar\'e disk model $\mathbb{D}\subset\mathbb{C}$, and the classical radial power weight of Bergman type arising in a number of problems in the theory of functions of a complex variable. Then we turn to the case of the Lobachevsky space $ \mathbb{H}^2 $. All these cases are embedded in one general scheme, but in each case the specific calculations differ and quite constructively illustrate the underlying geometry.

Finally, we come to a very interesting and illustrative example within a discussion on integral kernels on $\mathbb{R}^n$ that are homogeneous with respect to all invertible linear transformations. It appears that there are no such kernels for $n>2$ and instead there is just one such kernel up to a constant multiplier for $n=2.$ Further elaboration of the case $n=2$ provides a very specific integral operator (\ref{eq:UniqueIntOpn=2}), whose properties we hope to study in detail in a separate paper. There we also formulate an open question related to the multidimensional case $n>2$.

\section{Backgrounds, motivation and general thoughts}\label{Backgrounds, motivation and general thoughts}
\subsection{Classical integral operators with homogeneous kernel in Lebesgue spaces}

The one dimensional operators with homogeneous kernels
$$
Kf(x)=\int_0^\infty k(x,y)f(y)dy,
$$
where $k(x,y)$ is such that
$$
k(\lambda x,\lambda y)= \lambda^{-1}k(x,y), \ x,y\in \mathbb{R}_+^1,\ \lambda>0,
$$
appeared as counterparts to convolution operators in the following sense: these operators are invariant with respect to dilatations, not translations, in contrast to convolutions. At the same time, these operators are reduced to convolutions via the following isometries between $L^p(\mathbb{R}^1_+)$ and $L^p(\mathbb{R}^1),$ $1\leqslant p\leqslant \infty,$
\begin{eqnarray}
W_pf(t)&=&e^{-\frac{t}{p}}f(e^{-t}),\ \ t\in \mathbb{R},\\
W^{-1}_p g(\eta)&=& \eta^{-\frac{1}{p}}g(-\ln \eta),\ \ \eta\in\mathbb{R}^1_+.
\end{eqnarray}
Denote
$$
\kappa = \int_{\mathbb{R}^1_+}k(1,y)y^{-\frac{1}{p}}dy= \int_{\mathbb{R}^1_+}k(x,1)x^{-\frac{1}{q}}dx,\ \ \ \frac{1}{p}+\frac{1}{q}=1.
$$
The Hardy-Littlewood theorem states that
\begin{theorem}\label{th:HL}(See e.g.\cite{KarapetiantsSamko-book-eng}). Let $\kappa<\infty.$ The operator $K$ is bounded in $L^p(\mathbb{R}^1_+)$ with $\|K\|\leqslant \kappa,$ and $\|K\|=\kappa$ for non negative kernel $k(x,y).$
\end{theorem}
The  multidimensional operators
\begin{equation}
Kf(x)=\int_{\mathbb{R}^n} k(x,y)f(y)dy
\end{equation}
are such that
\begin{enumerate}
\item $k(\lambda x,\lambda y)=\lambda^{-n}k(x,y),\ \ \lambda>0,\ \ x,y\in\mathbb{R}^n;$
\item $k(\omega(x),\omega(y))=k(x,y)$ for any $\omega\in SO(n).$
\end{enumerate}
Some examples are
$$
k(x,y)=\frac{1}{|x|^\alpha |x-y|^{n-\alpha}},\ \ \ 0<\alpha<n;\,\,\,\,\ k(x,y)=\frac{1}{|x|^n+|y|^n}a\left(\frac{x\cdot y}{|x||y|}\right).
$$
\begin{remark}
In the literature it was customary to say that such kernels $k(x,y)$ depend only on $|x|, |y|$ and scalar product $x\cdot y.$ As we show below in the two dimensional case, this is an incomplete statement. The statement is true if assuming additionally invariance with respect to reflections. But the general form of such a kernel (without reflections) is in fact presented in (\ref{GeneralSolutionEuclid}). See also Remark \ref{Rem:HomogeneousOperatorsProblemEuclid}. In higher dimensions $n>2$ the discrepancy between the truly general form and merely functions expressible in scalar products will become even wider, for there are more degrees of freedom neglected in the latter case (functions of a point on the sphere versus functions of the projection of a point on a fixed axis).
\end{remark}

We end this section with multidimensional analog of Theorem \ref{th:HL}. Denoting
$$
\kappa =\int_{\mathbb{R}^n}k(e_1,y)y^{-\frac{n}{p}}dy,\ \ \ e_1=(1,0,\dots,0),
$$
we have
\begin{theorem}(See e.g.\cite{KarapetiantsSamko-book-eng}). Let $1\leqslant p\leqslant\infty.$ Let $\kappa<\infty.$ The operator $K$ is bounded in $L^p(\mathbb{R}^n)$ with $\|K\|\leqslant \kappa,$ and $\|K\|=\kappa$ for non negative kernel $k(x,y).$
\end{theorem}

The above mentioned results have weighted analogues. Such operators also have connections to Wiener-Hopf operators, and algebras of such operators together with equations involving such operators were studied thoroughly, see \cite{KarapetiantsSamko-book-eng} and references therein.

\subsection{Measure dilations}\label{SubSec:Measure dilations}

The choice of linear isotropic dilations (scalings) and rotations in the above standard definitions is natural from a metric space point of view: these are conformal transformations of the Riemannian manifold $\mathbb{R}^n$. The Lebesgue measure arises as the metric-induced measure and behaves appropriately with respect to these transformations, which allows one to make use of variable substitution techniques in the study of homogeneous integral operators, see Section \ref{Sec:Homogeneous integral operators} for precise definition of a homogeneous integral operator.

However, from a conceptual point of view an integral operator does not directly interact with the metric but rather with the measure. The class of transformations which transform the measure in a way similar to what conformal transformations do is far larger, and it is therefore both natural and interesting to study integral operators that are homogeneous with respect to such general, measure dilations, in the expectation that techniques analogous to the classical ones will become available due to the appropriate behaviour of the measure.

More precisely, the definition of a measure dilation (or simply dilation for the rest of the paper) is as follows.

\begin{definition}
Given a measure $\mu$ on a measurable space $M$, a measure dilation is a transformation $\varphi:M\to M$ in the appropriate category (measurable space automorphism for a measurable space, homeomorphism for a topological space, diffeomorphism for a manifold etc.), such that
$$
\mu(\varphi(A))=\lambda_\varphi\mu(A)\,\,\,\,\,\,\, (\mathrm{or\,\, simply}\,\,\, \mu\circ\varphi=\lambda_\varphi\mu)
$$
for all measurable sets $A\subset M$, where $\lambda_\varphi>0$ is a constant that depends only on $\varphi$.
\end{definition}

Note that in the archetypical setting of $M=\mathbb{R}^n$ in the definition above it would be customary to write $\lambda_\varphi^n$ in place of $\lambda_\varphi$, but since the dimension will not always be well-defined in our general setting, we do not follow that tradition.

It is easy to see that the set of all such $\varphi$ comprises a group, which we denote by $\mathrm{Dil}(M,\mu)$. The map $\varphi\mapsto\lambda_\varphi$ gives a group homomorphism $\lambda:\mathrm{Dil}(M,\mu)\to\mathbb{R}_+$ (character).

In the majority of cases the group $\mathrm{Dil}(M,\mu)$ will be too large to work with, and we will consider a subgroup $G\subset\mathrm{Dil}(M,\mu)$ instead. The character $\lambda$ will have a kernel $\ker\lambda\subset G$, a normal subgroup, and transformations $\varphi\in\ker\lambda$ will preserve the measure $\mu$, similar to rotations in the classical situation. However, a canonical factorization of $G$ (or $\mathrm{Dil}(M,\mu)$) into a product of $\ker\lambda$ and another subgroup $B$ will not always exist in full generality, and we will not be able to unambiguously speak of a subgroup $B\subset G$ of actual dilations, i.e., those performing the non-trivial dilations $\lambda_\varphi\neq1$. Therefore, we will simply call all $\varphi\in\mathrm{Dil}(M,\mu)$ dilations.

In case $G$ is a Lie group and $\lambda$ a is Lie group character, one can show that the semidirect factorization $G=\ker\lambda\rtimes\mathbb{R}$ does hold, and this may be used in certain applications.

In order to enhance the group theoretical language to be used in the paper, instead of writing $M\ni x\mapsto\varphi(x)\in M$ we will write $M\ni x\mapsto gx\ni M$ for $g\in G$, i.e., a left action of the group $G$ on the space $M$. We will assume that this action is transitive, i.e., for every $x,y\in M$ there exists $g\in G$ such that $gx=y$; otherwise the implications to be derived in this paper would hold in each $G$-orbit separately, independently of each other. Here we will not address the question whether or not a transitively acting $G$ exists. For instance, in the category of topological spaces, if $M$ is a rigid space then $\mathrm{Dil}(M,\mu)=\emptyset$ a priori. The basic definitions and facts about group actions on abstract sets, topological spaces and smooth manifolds can be found in monographs \cite{Lang02}, \cite{Folland15} and \cite{RudolphSchmidt13}, respectively.

Note that the particular case when $G\subset\ker\lambda$ corresponds to the $G$-invariant measure $\mu$ on $M$. Invariant measures are widely studied in non-commutative harmonic analysis and representation theory. Measure preserving transformations  are also relevant in the subject of incompressible flows.

Relaxing the strict measure preservation to allow a scalar factor, as in the definition given above, is thus the simplest natural step beyond the scope of invariant measures, and operators invariant with respect to a group of such transformations are thus an extension of usual convolution-type operators.

We will see below that in many cases the $L^p$-theory of such operators can be effectively reduced to that of convolution-type operators, see formulae (\ref{eq:Up}) and (\ref{RedConvOp}).

\subsection{Homogeneous integral kernels: strong and weak homogeneity conditions}

One of the main subjects of study in this paper is measurable functions $K$ which serve as integral kernels for integral operators $\operatorname{K}$, and satisfy certain homogeneity conditions with respect to a transitively acting group $G$ of dilations in the measure space $(M,\mu)$ (see Subsection \ref{SubSec:Measure dilations}).

We will operate with two slightly different homogeneity conditions applied to a kernel $K$.

\begin{definition}We say that a function (a kernel) $K$ satisfies weak or strong homogeneity condition provided the following statements hold, correspondingly:
\begin{itemize}
\item \textbf{Weak:}
\begin{equation}
\Bigl(\forall g\in G\Bigr)\Bigl(\mu^{\otimes2}-\mbox{a.e.}\,\,(x,y)\in M\times M\Bigr)\quad K(gx,gy)=\frac1{\lambda_g}K(x,y)\label{WeakHomKer}
\end{equation}
\item \textbf{Strong:}
\begin{equation}
\Bigl(\forall(x,y)\in M\times M\Bigr)\Bigl(\forall g\in G\Bigr)\quad K(gx,gy)=\frac1{\lambda_g}K(x,y)\label{StrongHomKer}
\end{equation}
\end{itemize}
Correspondingly, to the fulfilment of the above weak or strong conditions, we will call the corresponding kernel $K$ weakly homogeneous or strongly homogeneous. Later on we will omit the word "strongly" and use only the abbreviation "homogeneous kernel" for those kernels satisfying the strong type condition (\ref{StrongHomKer}), see Remark \ref{remark:From this moment on} for precise statement on this issue.
\end{definition}

It is clear that the strong condition implies the weak one. We will see later in Theorem \ref{HomOptoKerProp} that the invariance of the integral operator $\operatorname{K}$ with respect to the group of dilations $G$ (or homogeneity) is equivalent to the weak homogeneity condition (\ref{WeakHomKer}), but our most explicit results will require the stronger condition (\ref{StrongHomKer}), see Subsection \ref{SubSec: Homogeneous integral kernels}. In order to fill the gap between these two, weak and strong, notions of homogeneity, we may need to assume for instance, that $M$ is a decent topological space and $K$ is continuous $\mu$-a.e. But this issue will not be addressed in the present work.

\section{Dilations, homogeneous operators and homogeneous integral operators}\label{Dilations, homogeneous operators and homogeneous integral operators}
\subsection{Dilations}

Hereinafter $\mathbb{F}\in\{\mathbb{R},\mathbb{C}\}$ (this means either $\mathbb{F}=\mathbb{R}$ or $\mathbb{F}=\mathbb{C}$ ) will be fixed and for a measure space $(M,\mu)$, and  the symbol $L(M,\mu)$ will stand for the $\mathbb{F}$-vector space of measurable functions $f:M\to\mathbb{F}$ (more precisely, equivalence classes up to $\mu$-null subsets, as usual). We will work in a fixed subcategory of measure spaces, and $\Aut(M)$ will stand for the group of appropriate self-morphisms. Denote (recall that)
$$
\mathrm{Dil}(M,\mu)\doteq\{\varphi\in\Aut(M)\,\vline\quad\mu\circ\varphi=\lambda_\varphi\mu,\quad\lambda_\varphi>0\}.
$$
This gives a group character $\lambda:\mathrm{Dil}(M,\mu)\to\mathbb{R}_+$. Then $\ker\lambda\subset\mathrm{Dil}(M,\mu)$ is a normal subgroup. Let $G\subset\mathrm{Dil}(M,\mu)$ be a subgroup that acts transitively on $M$. The transitive action of $G$ by measurable space automorphisms allows us to represent $M$ as the homogeneous space $G/H$ with $G$-invariant measurable structure, where $H\subset G$ is a non-normal subgroup.

Henceforth we will deal with the measure space $(G/H,\mu)$ where $\mu(g.)=\lambda_g\mu(.)$ and $\lambda:G\to\mathbb{R}_+$ is a character. Here we will distinguish between two essentially different situations:

\begin{itemize}

\item \textbf{Case A:} $H\subset\ker\lambda$

\item \textbf{Case B:} $H\not\subset\ker\lambda$

\end{itemize}

In Case A, $\lambda$ gives a well-defined function $\lambda:G/H\to\mathbb{R}_+$. We will assume that this function is measurable with respect to the $G$-invariant measurable structure on $M=G/H$. This can be guaranteed if we equip $\mathrm{Dil}(M,\mu)$ with an appropriate measurable structure and require that $\lambda:\mathrm{Dil}(M,\mu)\to\mathbb{R}_+$ is measurable. This should be true in all cases of interest. Now we can introduce the measure $\tilde\mu$ on $G/H$ by setting
$$
d\tilde\mu(x)=\frac1{\lambda_x}d\mu(x),\quad d\tilde\mu(gx)=\frac1{\lambda_{gx}}d\mu(gx)=d\tilde\mu(x),\quad\forall x\in G/H,
$$
and we see that $\tilde\mu$ is $G$-invariant.

We can introduce for $p>0$ the linear operator $\operatorname{U}_p:L(G/H,\mu)\to L(G/H,\mu)$ by setting
\begin{equation}\label{eq:Up}
\operatorname{U}_pf(x)\doteq\lambda_x^{\frac1p}f(x),\quad\forall x\in G/H,\quad\forall f\in L(G/H,\mu).
\end{equation}
This will give a unitary operator $L^p(G/H,\mu)\to L^p(G/H,\tilde\mu),$ and recall that all the above said is valid in the Case A.

In Case B we will not be able to reduce $\mu$ to an invariant measure, and this will have implications for homogeneous integral operators to be discussed below, see Remark \ref{SingKerRemark}.

Let us note that at this level of generality, any abstract groups $G$, $H\subset G$ and any character $\lambda:G\to\mathbb{R}_+$ can figure in this construction. Indeed, take a left-invariant measure $\mu_L$ on $G$ (for instance, the left Haar measure on $G$ taken with discrete topology) and define the measure $\tilde\mu$ on $G$ by $\frac{d\mu_L(g)}{d\tilde\mu(g)}=\lambda_g$ for all $g\in G$. Then define the measure $\mu$ on $G/H$ by $\mu=\tilde\mu\circ\operatorname{q}_H^{-1}$, where $\operatorname{q}_H:G\to G/H$ is the standard quotient map which is measurable and $G$-equivariant. It is easy to check that $G\subset\mathrm{Dil}(G/H,\mu)$.

\subsection{The geometry of dilations}

Define the set
$$
\mathcal{X}^\lambda_H\doteq\left\{aH\,\vline\quad aHa^{-1}\cap H\not\subset\ker\lambda\right\}\subset G/H.
$$
It can be described also as
$$
\mathcal{X}^\lambda_H=\left\{x\in G/H\,\vline\quad(\exists h\in H)\,hx=x\quad\wedge\quad\lambda_h\neq1\right\}.
$$
If we denote temporarily by $\phi:H\times G/H\to\mathbb{R}_+\times G/H\times G/H$ the measurable map $(h,x)\mapsto(\lambda_h,hx,x)$, then we can see that $\mathcal{X}^\lambda_H$ is the projection onto the second component of the preimage $\phi^{-1}(Y)$ of the measurable set
$$
Y=(\mathbb{R}_+\setminus\{1\})\times\mathrm{diag}(G/H\times G/H)\subset\mathbb{R}_+\times G/H\times G/H.
$$
This shows that $\mathcal{X}^\lambda_H$ is a measurable set.

It is clear that in Case A above $\mathcal{X}^\lambda_H=\emptyset$. In Case B we have $\mathcal{X}^\lambda_H\neq\emptyset$, because $\id H\in\mathcal{X}^\lambda_H$ holds automatically. In fact, it is easy to check that $\mathcal{X}^\lambda_H$ is invariant under the action of $\mathrm{N}(H)$ and the transformation $aH\mapsto a^{-1}H$,
$$
\mathcal{X}^\lambda_H=\mathrm{N}(H)\mathcal{X}^\lambda_H,\quad\bigl(\forall aH\in G/H\bigr)a^{-1}H\in\mathcal{X}^\lambda_H,
$$
where $\mathrm{N}(H)$ is the normalizer of $H$ in $G$. In particular, $\mathrm{N}(H)H\subset\mathcal{X}^\lambda_H$ is always true, and the action of $H$ on $\mathrm{N}(H)H$ is trivial. In many practical situations we will have $\mathrm{N}(H)H=\mathcal{X}^\lambda_H$ and/or $\mu(\mathcal{X}^\lambda_H)=0$.

Denote by $M_*\doteq G/H\setminus\mathcal{X}^\lambda_H$ the ``regular'' part of $G/H$, which is again a $\mathrm{N}(H)$-invariant and $aH\mapsto a^{-1}H$-invariant measurable set. Let $\operatorname{p}_H:G/H\to H\backslash G/H$ be the standard quotient map, which induces by pushforward a measurable structure on $H\backslash G/H$. We will write $H\backslash M_*\doteq\operatorname{p}_H(M_*)$ and $H\backslash\mathcal{X}^\lambda_H\doteq\operatorname{p}_H(\mathcal{X}^\lambda_H)$, which are measurable sets such that $H\backslash G/H=H\backslash M_*\cup H\backslash\mathcal{X}^\lambda_H$. Note also that $\operatorname{p}_H|_{\mathrm{N}(H)H}$ is injective.

It is tempting to define a measure $\nu\doteq\mu\circ\operatorname{p}_H^{-1}$ on $H\backslash G/H$. However, in Case B we will discover that $\nu(A)\in\{0,+\infty\}$ for every measurable set $A\subset H\backslash G/H$, which renders the measure $\nu$ less than useful. Nonetheless, the notion of a $\nu$-null set in $H\backslash G/H$ thus defined is always sensible.

Consider the relation $\simeq$ on the space $M_*\times\mathbb{F}$ given by
$$
(aH,\alpha)\simeq(bH,\beta)\quad\Leftrightarrow\quad\bigl(\exists h\in H\bigr)\,b=ha\,\wedge\,\alpha=\lambda_h\beta,\quad\forall(aH,\alpha),(bH,\beta)\in M_*\times\mathbb{F}.
$$
It is easily checked that $\simeq$ is an equivalence relation. Denote $\mathcal{T}\doteq M_*/\simeq$ and let $\operatorname{q}_\simeq:M_*\times\mathbb{F}\to\mathcal{T}$ be the quotient map. We will write $[aH,\alpha]_\simeq\doteq\operatorname{q}_\simeq((aH,\alpha))$ for $(aH,\alpha)\in M_*\times\mathbb{F}$. Define the map $\pi:\mathcal{T}\to H\backslash M_*$ by $\pi([aH,\alpha]_\simeq)=HaH$ for all $[aH,\alpha]_\simeq\in\mathcal{T}$, which is easily seen to be surjective. Moreover, if we define an $\mathbb{F}$-module structure on $\mathcal{T}$ by
$$
\beta\cdot[aH,\alpha]_\simeq\doteq[aH,\beta\alpha]_\simeq,\quad\forall\beta\in\mathbb{F},\quad\forall[aH,\alpha]_\simeq\in\mathcal{T},
$$
then $\operatorname{q}_\simeq$ is easily seen to be $\mathbb{F}$-linear, and
$$
\pi(\beta\cdot[aH,\alpha]_\simeq)=\pi([aH,\alpha]_\simeq),\quad\forall\beta\in\mathbb{F},\quad\forall[aH,\alpha]_\simeq\in\mathcal{T}.
$$
Another way to see this would be to define a left $H$-action on the $\mathbb{F}$-line bundle $M_*\times\mathbb{F}$ by
$$
h\cdot(aH,\alpha)\doteq\left(haH,\frac{\alpha}{\lambda_h}\right),\quad\forall h\in H,\quad\forall(aH,\alpha)\in M_*\times\mathbb{F},
$$
and set $\mathcal{T}=H\backslash(M_*\times\mathbb{F})$. Thus, $\mathcal{T}$ is an $\mathbb{F}$-line bundle over $H\backslash M_*$. Moreover, $\mathcal{T}$ has a measurable structure inherited from $M_*\times\mathbb{F}$, and the projection $\pi$ is measurable. If $L(M_*,\mu)$ and $L(\mathcal{T},\nu)$ stand for the $\mathbb{F}$-vector spaces of measurable sections in the line bundles $M_*\times\mathbb{F}$ and $\mathcal{T}$ up to $\mu$-a.e. and $\nu$-a.e. zero sections, respectively, then we can define the following $\mathbb{F}$-linear map $\wp:L(\mathcal{T},\nu)\to L(M_*,\mu)$,
$$
\wp[s](aH)\doteq(aH,\alpha),\quad s(HaH)=[aH,\alpha]_\simeq,\quad\forall aH\in M_*,\quad\forall s\in L(\mathcal{T},\nu).
$$
That this map is well-defined follows by the transversality of the $H$-action on $M_*\times\mathbb{F}$.

Introduce the following class of functions,
$$
\mathcal{F}^\lambda_H\doteq\left\{F\in L(G/H,\mu)\,\vline\quad\Bigl(\forall x\in G/H\Bigr)\Bigl(\forall h\in H\Bigr) F(hx)=\frac1{\lambda_h}F(x)\right\}.
$$
It is clear that if $F\in\mathcal{F}^\lambda_H$ then $F|_{\mathcal{X}^\lambda_H}=0$ and $F|_{M_*}\in\wp(L(\mathcal{T},\nu))$. Thus,
\begin{equation}
\mathcal{F}^\lambda_H=\left\{F\in L(G/H,\mu)\,\vline\quad F=F_*\cup 0,\quad F_*\in\wp(L(\mathcal{T},\nu))\right\}.\label{FlambdaHDesc}
\end{equation}

\subsection{Homogeneous operators}\label{Sec:Homogeneous operators}

Consider the so-called quasiregular representation of $G$ on functions over $G/H$ by pullback,
$$
\operatorname{L}_gf(x)=f(g^{-1}x),\quad\forall x\in G/H,\quad\forall g\in G,\quad\forall f\in\mathbb{F}^M.
$$
Since for every measurable function $f:M\to\mathbb{F}$ we have
$$
\mu((\operatorname{L}_gf)^{-1}(0))=\mu(g\cdot f^{-1}(0))=\lambda_g\mu(f^{-1}(0)),\quad\forall g\in G,
$$
the image of a $\mu$-a.e. vanishing function in $\operatorname{L}_g$ vanishes $\mu$-a.e. Thus, for every $g\in G$ this gives a linear operator $\operatorname{L}_g:L(G/H,\mu)\to L(G/H,\mu)$.

Note that in Case A, for every $g\in G$ the operator $\operatorname{L}_g$ satisfies
$$
\operatorname{U}_p\operatorname{L}_g=\lambda_g^{\frac1p}\operatorname{L}_g\operatorname{U}_p,\quad\forall g\in G,\quad\forall p\in\mathbb{R}_+,
$$
where $\operatorname{U}_p$ is given by (\ref{eq:Up}). Also, the operator $\operatorname{L}_g:L^2(G/H,\tilde\mu)\to L^2(G/H,\tilde\mu)$ is unitary, $g\in G$.

\begin{definition}\label{def;HO}
Let $\mathrm{D}\subset L(G/H,\mu)$ be a vector subspace. Consider a linear operator $\operatorname{K}:\mathrm{D}\to L(G/H)$. We will say that $\operatorname{K}$ is homogeneous if $\mathrm{D}$ is $\operatorname{L}_g$-invariant,
$$
\operatorname{L}_g(\mathrm{D})\subset\mathrm{D},
$$
and $\operatorname{K}$ commutes with $\operatorname{L}_g$ on $\mathrm{D}$ for all $g\in G$,
$$
\operatorname{K}\operatorname{L}_gf=\operatorname{L}_g\operatorname{K}f,\quad\forall f\in\mathrm{D},\quad\forall g\in G.
$$
\end{definition}

In Case A, for $p>0$ denote $\widetilde{\mathrm{D}}_p\doteq\operatorname{U}_p(\mathrm{D})$ and observe that $\widetilde{\mathrm{D}}_p$ is $\operatorname{L}_g$-invariant,
$$
\operatorname{L}_g(\widetilde{\mathrm{D}}_p)=\operatorname{L}_g\operatorname{U}_p(\mathrm{D})=\lambda_g^{-\frac1p}\operatorname{U}_p\operatorname{L}_g(\mathrm{D})\subset\widetilde{\mathrm{D}}_p,\quad\forall g\in G,\quad\forall p\in\mathbb{R}_+.
$$
Define for $p>0$ the unitarily transformed operator $\widetilde{\operatorname{K}}_p=\operatorname{U}_p\operatorname{K}\operatorname{U}_p^{-1}:\widetilde{\mathrm{D}}_p\to L(G/H,\tilde\mu)$, which satisfies
$$
\operatorname{L}_g\widetilde{\operatorname{K}}_pf=\operatorname{L}_g\operatorname{U}_p\operatorname{K}\operatorname{U}_p^{-1}f=\lambda_g^{-\frac1p}\operatorname{U}_p\operatorname{L}_g\operatorname{K}\operatorname{U}_p^{-1}f=\lambda_g^{-\frac1p}\operatorname{U}_p\operatorname{K}\operatorname{L}_g\operatorname{U}_p^{-1}f
$$
\begin{equation}
=\operatorname{U}_p\operatorname{K}\operatorname{U}_p^{-1}\operatorname{L}_gf=\widetilde{\operatorname{K}}_p\operatorname{L}_gf,\quad\forall f\in\widetilde{\mathrm{D}},\quad\forall g\in G,\quad\forall p\in\mathbb{R}_+.\label{RedConvOp}
\end{equation}
Thus, in Case A the operator $\widetilde{\operatorname{K}}_p$ is a general $G$-invariant (convolution-type) operator on $G/H$ studied in non-commutative harmonic analysis.

Note that in Case B such a reduction to $G$-invariant operators is not possible. In some sense, homogeneous operators represent a generalization of convolution-type operators, and a more explicit relation between the two will be seen in what follows.

\subsection{Homogeneous integral operators}\label{Sec:Homogeneous integral operators}

An operator $\operatorname{K}:\mathrm{D}\to L(G/H,\mu)$, with $\mathrm{D}\subset L(G/H,\mu)$ as before, will be called an integral operator if there is a measurable kernel function $K\in L(G/H\otimes G/H,\mu^{\otimes2})$ such that
\begin{equation}
\operatorname{K}f(x)=\int\limits_{G/H}K(x,y)f(y)d\mu(y),\quad\mu-\mbox{a.e.}\,\,x\in G/H,\quad\forall f\in\mathrm{D}.\label{TDef}
\end{equation}
Note that if the integral kernel $K$ is singular, as is the case in Example \ref{GL2HomKerExample} below, then the integral in formula (\ref{TDef}) may not converge in the usual, absolute sense. In that case, in order to make sense of $\mathrm{K}$ as a singular integral operator, one needs to specify a precise conditional convergence scheme.
\begin{definition}Let $\mathrm{D}\subset L(G/H,\mu)$ be a vector subspace.
A homogeneous integral operator $\operatorname{K}:\mathrm{D}\to L(G/H)$ is an integral operator of the form (\ref{TDef}) which is also a homogeneous operator as defined above in Definition \ref{def;HO}.
\end{definition}
We will need a separation condition for the domain $\mathrm{D}$, namely, that there exists a sequence $\{f_k\}_{k=1}^\infty\subset\mathrm{D}$ such that
\begin{equation}
(\forall F\in L(G/H,\mu))\left[(\forall k\in\mathbb{N})\int\limits_{G/H}F(x)f_k(x)d\mu(x)=0\right]\quad\Rightarrow\quad F=0.\label{DSepPoints1}
\end{equation}
It will be shown that this condition is easily satisfied in the most relevant situations.

\begin{theorem}\label{HomOptoKerProp} Assume that $\mathrm{D}$ separates points in $L(G/H,\mu)$ as in (\ref{DSepPoints1}) and $\operatorname{L}_g(\mathrm{D})\subset\mathrm{D}$ for all $g\in G$. Then the integral operator $\operatorname{K}$ with integral kernel $K$ defined by formula (\ref{TDef}) is homogeneous if and only if $K$ satisfies the weak homogeneity condition (\ref{WeakHomKer}), which in the case under consideration takes the form:
\begin{equation}
\Bigl(\forall g\in G\Bigr)\Bigl(\mu^{\otimes2}-\mbox{a.e.}\,\,(x,y)\in G/H\times G/H\Bigr)\quad K(gx,gy)=\frac1{\lambda_g}K(x,y).\label{KEq}
\end{equation}
\end{theorem}
\begin{proof}For every $f\in\mathrm{D}$ let $X_f\subset G/H$ be the $\mu$-null set on which the equality (\ref{TDef}) does not hold.

$\Rightarrow$\quad Assume that $\operatorname{K}$ is homogeneous,
$$
\Bigl(\forall g\in G\Bigr)\Bigl(\forall f\in \mathrm{D}\Bigr)\Bigl(\exists V_{g,f}\subset G/H\Bigr)\quad\mu(V_{g,f})=0\quad\wedge\quad\left(\forall x\not\in V_{g,f}\right)\,\operatorname{L_g}\operatorname{K}f(x)=\operatorname{K}\operatorname{L_g}f(x).
$$
Define
$$
V_g\doteq\bigcup_{k=1}^\infty V_{g,\operatorname{L}_{g^{-1}}f_k},\quad U_g\doteq\bigcup_{k=1}^\infty gX_{\operatorname{L}_{g^{-1}}f_k},\quad W_g\doteq\bigcup_{k=1}^\infty X_{f_k},\quad\forall g\in G,
$$
where $\{f_k\}_{k=1}^\infty\subset\mathrm{D}$ are from (\ref{DSepPoints1}). Then by $\sigma$-additivity of $\mu$,
$$
\mu(V_g)\le\sum_{k=1}^\infty\mu(V_{g,\operatorname{L}_{g^{-1}}f_k})=0,\quad\mu(U_g)\le\sum_{k=1}^\infty\mu(gX_{\operatorname{L}_{g^{-1}}f_k})=0,
$$
$$
\mu(W_g)\le\sum_{k=1}^\infty\mu(X_{f_k})=0,\quad\forall g\in G.
$$
Take an $x\in G/H\setminus(V_g\cup U_g\cup W_g)$. For all $k\in\mathbb{N}$ we have
$$
\operatorname{L}_g\operatorname{K}\operatorname{L}_{g^{-1}}f_k(x)=\operatorname{K}\operatorname{L}_{g^{-1}}f_k(g^{-1}x)\overset{x\not\in gX_{\operatorname{L}_{g^{-1}}f_k}}{=\!=}\int\limits_{G/H}K(g^{-1}x,y)\operatorname{L}_{g^{-1}}f_k(y)d\mu(y)
$$
$$
=\int\limits_{G/H}K(g^{-1}x,g^{-1}z)\lambda_{g^{-1}}f_k(z)d\mu(z)\overset{x\not\in V_{g,\operatorname{L}_{g^{-1}}f_k}}{=\!=}\operatorname{K}f_k(x)\overset{x\not\in X_{f_k}}{=\!=}\int\limits_{G/H}K(x,z)f_k(z)d\mu(z),
$$
whence
$$
\int\limits_{G/H}\left(K(g^{-1}x,g^{-1}z)\lambda_{g^{-1}}-K(x,z)\right)f_k(z)d\mu(z)=0,\quad\forall k\in\mathbb{N}.
$$
By separation property (\ref{DSepPoints1}),
$$
\Bigl(\exists Y_g^x\subset G/H\Bigr)\quad\mu(Y_g^x)=0\quad\wedge\quad\left(\forall y\not\in Y_g^xG\right)\,K(g^{-1}x,g^{-1}y)=\frac1{\lambda_{g^{-1}}}K(x,y).
$$
Denote $Y_g\doteq\{(x,y)\in G/H\times G/H\,|\quad y\in Y_g^x\}$, whence
$$
\mu^{\otimes2}(Y_g)=\int\limits_{G/H}\mu(Y_g^x)d\mu(x)=0.
$$
It follows that
$$
\mu^{\otimes2}\Bigl(\bigl((V_g\cup U_g\cup W_g)\times G/H\bigr)\bigcup Y_g\Bigr)=0,
$$
so that the equation in the statement (\ref{KEq}) holds $\mu^{\otimes2}$-a.e.

$\Leftarrow$\quad Assume that
$$
\Bigl(\forall g\in G\Bigr)\Bigl(\exists Z_g\subset G/H\times G/H\Bigr)\quad\mu^{\otimes2}(Z_g)=0\quad\wedge\quad\left(\forall (x,y)\not\in Z_g\right)\, K(gx,gy)=\frac1{\lambda_g}K(x,y).
$$
Denote
$$
Z_g^x\doteq\{y\in G/H\,\vline\quad(x,y)\in Z_g\},\quad Y_g\doteq\{x\in G/H\,\vline\quad\mu(Z_g^x)>0\},\quad\forall g\in G,\quad\forall x\in G/H,
$$
and observe that
$$
\mu^{\otimes2}(Z_g)=\int_{Y_g}\mu(Z_g^x)d\mu(x)=0,
$$
so that $\mu(Y_g)=0$ for all $g\in G$. Fix a $g\in G$ and an $f\in\mathrm{D}$, and take any $x\in G/H\setminus(gX_f\cup X_{\operatorname{L}_gf}\cup Y_{g^{-1}})$. Then
$$
\operatorname{L_g}\operatorname{K}f(x)=\operatorname{K}f(g^{-1}x)\overset{x\not\in gX_f}{=\!=}\int\limits_{G/H}K(g^{-1}x,y)f(y)d\mu(y)=\int\limits_{G/H}K(g^{-1}x,g^{-1}z)f(g^{-1}z)d\mu(g^{-1}z)
$$
\begin{equation}
=\int\limits_{G/H}K(g^{-1}x,g^{-1}z)\operatorname{L_g}f(z)\lambda_{g^{-1}}d\mu(z)\overset{x\not\in Y_{g^{-1}}}{=\!=}\int\limits_{G/H}K(x,z)\operatorname{L_g}f(z)d\mu(z)\overset{x\not\in X_{\operatorname{L}_gf}}{=\!=}\operatorname{K}\operatorname{L_g}f(x).\label{KernelNec}
\end{equation}
But
$$
\mu(gX_f\cup X_{\operatorname{L}_gf}\cup Y_{g^{-1}})\le\mu(gX_f)+\mu(X_{\operatorname{L}_gf})+\mu(Y_{g^{-1}})=0,
$$
which means that (\ref{KernelNec}) holds for $\mu$-a.e. $x\in G/H$, and thus $\operatorname{L_g}\operatorname{K}f=\operatorname{K}\operatorname{L_g}f$ in the sense of $L(G/H,\mu)$. $\Box$
\end{proof}

\begin{remark} Note that the separation property (\ref{DSepPoints1}) is used only in the sufficiency implication of \textnormal{Theorem \ref{HomOptoKerProp}}. Therefore the necessity part in \textnormal{Theorem \ref{HomOptoKerProp}} remains true without assuming the separation property (\ref{DSepPoints1}).
\end{remark}

\subsection{On a sufficient condition for the separation property}

In this section we want to give a reasonable sufficient condition for the separation property (\ref{DSepPoints1}) to hold for the domain $\mathrm{D}\subset L(G/H,\mu)$ of an integral operator $\operatorname{K}$ with locally integrable integral kernel $K$. We will do this for the most common case where $G/H$ is a locally compact, Hausdorff, second countable space with its Borel measurable structure. This is always true if $G/H$ is a manifold with countably many connected components, or when $G$ is a locally compact, Hausdorff, second countable topological group, and $H\subset G$ a closed subgroup.

\begin{lemma}\label{SepLemma} Let $X$ be a locally compact, Hausdorff, second countable space, and let $\mu$ be a Radon measure on $X$. There exists a sequence $\{f_k\}_{k=1}^\infty\subset C_c(X,\mathbb{R})$ of compactly supported continuous real functions such that
$$
(\forall F\in L(X,\mu))\left[(\forall k\in\mathbb{N})\int\limits_XF(x)f_k(x)d\mu(x)=0\right]\quad\Rightarrow\quad F=0.
$$
\end{lemma}
\begin{proof} We first note that $X$ is $\sigma$-compact, and therefore we can arrange for a sequence $\{K_n\}_{n=1}^\infty$ of compact sets such that
$$
(\forall n\in\mathbb{N})\,K_n\Subset\mathring{K}_{n+1},\quad\bigcup_{n=1}^\infty K_n=X.
$$
Then $C_c(X,\mathbb{R})$ has the inductive limit topology
$$
C_c(X,\mathbb{R})=\bigcup_{n=1}^\infty C_0(K_n,\mathbb{R}),\quad C_0(K_n,\mathbb{R})=\left\{f\in C_c(X,\mathbb{R})\,\vline\quad\supp f\subset K_n\right\},\quad\forall n\in\mathbb{N}.
$$
Each $C_0(K_n,\mathbb{R})$ is a subspace of the separable metric space $C(K_n,\mathbb{R})$ with uniform norm, and is therefore separable. Let $\{f_k\}_{k=1}^\infty$ be the union of countable dense sets in all $C_0(K_n,\mathbb{R})$, which is thus dense in $C_c(X,\mathbb{R})$. Now suppose that $F\in L(X,\mu)$ is such that
$$
(\forall k\in\mathbb{N})\int\limits_XF(x)f_k(x)d\mu(x)=0.
$$
Then $F\in L^1_{\mathrm{loc}}(X,\mu)$. For otherwise $\exists x_0\in X$ such that $\forall\Delta\subset X$ open,
$$
x_0\in\Delta\quad\Rightarrow\quad\int\limits_\Delta|F(x)|d\mu(x)=+\infty.
$$
By density, $\exists k_0\in\mathbb{N}$ such that $f_{k_0}(x_0)\neq 0$. By continuity, $\exists\Delta_0\subset X$ open such that
$$
x_0\in\Delta_0\quad\wedge\quad(\forall x\in\Delta_0)\,|f_{k_0}(x)|>\frac12|f_{k_0}(x_0)|>0.
$$
But then
$$
\int\limits_{\Delta_0}|F(x)||f_{k_0}(x)|d\mu(x)>\frac12|f_{k_0}(x_0)|\int\limits_{\Delta_0}|F(x)|d\mu(x)=+\infty,
$$
in contradiction to $Ff_{k_0}\in L^1(X,\mu)$. Now that $F\in L^1_{\mathrm{loc}}(X,\mu)$, we have that the linear functional
$$
C_c(X,\mathbb{R})\ni f\mapsto\int\limits_XF(x)f(x)d\mu(x)\in\mathbb{C}
$$
is continuous, and since it vanishes on the dense set $\{f_k\}_{k=1}^\infty$, it is identically zero. That this implies $F=0$ ($\mu$-a.e.) can be shown by classical arguments. $\Box$
\end{proof}

Now if the space $G/H$ is locally compact, Hausdorff and second countable, and if the integral kernel $K\in L^1_{\mathrm{loc}}(G/H\times G/H,\mu^{\otimes 2})$ (or at least $K(x,.)\in L^1_{\mathrm{loc}}(G/H,\mu)$ for $\mu$-a.e. $x\in G/H$) then there is no loss of generality in assuming that $C_c(G/H,\mathbb{R})\subset\mathrm{D}$. Therefore the separation property (\ref{DSepPoints1}) can be satisfied with the sequence $\{f_k\}_{k=1}^\infty$ from Lemma \ref{SepLemma}.

\subsection{Homogeneous (strongly homogeneous) integral kernels}\label{SubSec: Homogeneous integral kernels}

Let $K\in L(G/H\times G/H,\mu^{\otimes2})$ be the integral kernel of an integral operator $\operatorname{K}:\mathrm{D}\to L(G/H,\mu)$ as in (\ref{TDef}). We saw that the homogeneity of the operator $\operatorname{K}$ is equivalent to the weak homogeneity condition (\ref{WeakHomKer}) for the kernel $K$. But in order to derive our most important and applicable results we will need the stronger homogeneity condition (\ref{StrongHomKer}).
\begin{remark}\label{remark:From this moment on}
From this moment on, we will drop the word "strong" and call the kernels which satisfy (\ref{StrongHomKer}) homogeneous kernels, since our next steps are associated with just such kernels. Consider a modification of the weak homogeneity condition (\ref{WeakHomKer}) where the order of the two quantifiers is reversed,
\begin{equation*}
\Bigl(\mu^{\otimes2}-\mbox{a.e.}\,\,(x,y)\in G/H\times G/H\Bigr)\Bigl(\forall g\in G\Bigr)\quad K(gx,gy)=\frac1{\lambda_g}K(x,y).
\end{equation*}
If $Z\subset G/H\times G/H$ is the $\mu^{\otimes2}$-null set on which this homogeneity property does not hold, then it is easy to see that both $Z$ and $G/H\times G/H\setminus Z$ are $G$-invariant. In particular, assuming $K|_Z=0$ does not alter $K$ in the setting of $L(G/H\times G/H,\mu^{\otimes2})$ on one hand, and provides
\begin{equation}
\Bigl(\forall(x,y)\in G/H\times G/H\Bigr)\Bigl(\forall g\in G\Bigr)\quad K(gx,gy)=\frac1{\lambda_g}K(x,y)\label{KEq1}
\end{equation}
on the other hand. This will therefore be our definition of a homogeneous integral kernel. Note that formula (\ref{KEq1}) is the same as (\ref{StrongHomKer}) specified for $M=G/H$.
\end{remark}
In order to avoid misleading let us also clarify for the reader's convenience the following issue.
\begin{remark}
Regarding the above said, let us recall that a weakly homogeneous kernel corresponds to a homogeneous (integral) operator (Theorem \ref {HomOptoKerProp}), and homogeneous kernels (previously referred to as strongly homogeneous kernels) generate a subclass of integral operators which of course are homogeneous operators in the sense of our definition. Thus, according to our definition, an integral operator with a homogeneous kernel is necessarily a homogeneous integral operator, but the converse is not true in general.
\end{remark}

That (\ref{WeakHomKer}) does not always imply (\ref{StrongHomKer}) can be seen in the following example.
\begin{example} Let $G=\mathbb{R}$, $H=\{0\}$, $d\mu(x)=dx$ and $\lambda=\id$. Consider $K:\mathbb{R}^2\to\mathbb{R}$ given by
$$
K(x,y)=\begin{cases}
0 & \mbox{if}\quad y-\lfloor y\rfloor=x-y-\lfloor x-y\rfloor\\
1 & \mbox{else}
\end{cases}
$$
For every $r\in\mathbb{R}$ set $K_r(x,y)\doteq K(x+r,y+r)$ for all $(x,y)\in\mathbb{R}^2$. Then $\mu(F^{-1}(0))=\mu(F_r^{-1}(0))=0$, thus $F(x+r,y+r)=F(x,y)$ for $\mu$-a.e. $(x,y)\in\mathbb{R}$, as in (\ref{KEq}). However, for a fixed $(x,y)\in\mathbb{R}^2$, the condition $F(x+r,y+r)=F(x,y)$ for all $r\in\mathbb{R}$ implies
$$
\Bigl(\forall r\in \mathbb{R}\Bigr)\quad y+r-\lfloor y+r\rfloor=y-\lfloor y\rfloor\quad\vee\quad y+r-\lfloor y+r\rfloor\neq x-y-\lfloor x-y\rfloor
$$
$$
\Rightarrow\quad\Bigl(\forall r\in\mathbb{R}\setminus\mathbb{Z}\Bigr)\quad y+r-\lfloor y+r\rfloor\neq x-y-\lfloor x-y\rfloor,
$$
which can hold only if $x-2y\in\mathbb{Z}$ (otherwise take $r=x-2y$). Thus, the condition $F(x+r,y+r)=F(x,y)$ for all $r\in\mathbb{R}$ holds only on the zero measure set where $x-2y\in\mathbb{Z}$, which is as far as possible from the statement (\ref{KEq1}).
\end{example}

However, conditions (\ref{WeakHomKer}) and (\ref{StrongHomKer}), in the current context written in the form (\ref{KEq}) and (\ref{KEq1}), respectively, are expected to be equivalent under practically reasonable assumptions on $K$.

The next theorem provides a general form for homogeneous (strongly homogeneous) kernels. Further in Section \ref{Homogeneous integral kernels and homogeneous integral operators over some domains}, in each specific case we will supply more explicit representations of this general form.

\begin{theorem} A function $K\in L(G/H\otimes G/H,\mu)$ satisfies the strong homogeneity condition (\ref{StrongHomKer}) if and only if
\begin{equation}
K(aH,bH)=\frac{F(a^{-1}bH)}{\lambda_a},\quad\forall(aH,bH)\in G/H\times G/H,\quad F\in\mathcal{F}^\lambda_H.\label{KProp}
\end{equation}
\end{theorem}
\begin{proof}Let $K$ satisfy condition (\ref{StrongHomKer}), and here we refer to its form as given in (\ref{KEq1}). Define $F\in L(G/H,\mu)$ by setting
$$
F(aH)\doteq K(\id H,aH),\quad\forall aH\in G/H.
$$
For every $(aH,bH)\in G/H\times G/H$ choose $g=a^{-1}$ to find that
$$
K(aH,bH)=\frac1{\lambda_a}K(\id H,a^{-1}bH)=\frac1{\lambda_a}F(a^{-1}bH).
$$
Moreover,
$$
F(haH)=K(\id H,haH)=K(h\id H,haH)=\frac{K(\id H,aH)}{\lambda_h}=\frac{F(aH)}{\lambda_h},\quad\forall aH\in G/H,\quad\forall h\in H,
$$
showing that $F\in\mathcal{F}^\lambda_H$. Conversely, let $K$ satisfy (\ref{KProp}). Then
$$
K(gaH,gbH)=\frac{F(a^{-1}bH)}{\lambda_{ga}}=\frac{K(aH,bH)}{\lambda_g},\quad\forall(aH,bH)\in G/H\times G/H,\quad\forall g\in G,
$$
as desired. $\Box$
\end{proof}

\begin{remark}\label{SingKerRemark} In this light, the description (\ref{FlambdaHDesc}) implies that homogeneous kernels are generally expected to exhibit exceptional or singular behaviour on $\mathcal{X}^\lambda_H$ in Case B. Also, the dimension of the space of all homogeneous kernels is the cardinality of $H\backslash M_*$. In particular, if $H\backslash M_*$ is a single point then there is up to a scalar factor a unique homogeneous kernel.
\end{remark}

\section{Homogeneous integral kernels and homogeneous integral operators over some domains}\label{Homogeneous integral kernels and homogeneous integral operators over some domains}

Given a concrete measure space $(M,\mu)$, there will generally be infinitely many possible choices of transitively acting $G\subset\mathrm{Dil}(M,\mu)$ to work with, but describing any one of them explicitly may not be an easy job. It will often be associated with solving functional or differential equations resulting from the condition $\mu\circ\varphi=\lambda_\varphi\mu$. Once $G$ is fixed (then so is $\lambda$), the passage from original local coordinates on $M$ to those adapted to the homogeneous space structure $G/H$ may be another challenge, which is necessary in order to translate the general results about homogeneous kernels to the original local coordinates. Often these two tasks can be accomplished simultaneously by finding a measure space isomorphism with another measure space $(\tilde M,\tilde\mu)$, for which the problem has been solved before.

Below we will discuss typical examples where both dilations and homogeneous integral kernels can be found explicitly.

\subsection{Homogeneous integral kernels on the cylinder $\mathbb{R}\times\mathbb{T}$}

Let us consider the case of the cylinder $M=G=\mathbb{R}\times\mathbb{T}$ with global coordinates $g=(z,\theta)$ and composition law
\begin{equation}
(z_x,\theta_x)(z_y,\theta_y)=(z_x+z_y,\theta_x+\theta_y\mod2\pi).
\end{equation}
Every character $\lambda$ of $G$ is then the product $\lambda=\lambda^{(1)}\lambda^{(2)}$ of a character $\lambda^{(1)}$ of $\mathbb{R}$ and a character $\lambda^{(2)}$ of $\mathbb{T}$. The latter is a compact group and cannot have non-trivial real characters, so that $\lambda^{(2)}=1$. On the other hand, the group $\mathbb{R}$ has only real characters of the form
\begin{equation}
\lambda^{(1)}_z=e^{\omega z},\quad\omega\in\mathbb{R}_{+}\cup\{0\}.
\end{equation}
For a non-trivial character, the exponent $\omega>0$ can be fixed to $\omega=2$ (for conformity with standard notations) by coordinate transformation $z\to\omega z/2$ (which is a group automorphism), so that without loss of generality we can take
\begin{equation}
\lambda_{(z,\theta)}=e^{2z}.
\end{equation}
The measure dilated according to this character is $d\mu(z,\theta)=e^{2z}dzd\theta$. Now the strong homogeneity condition (\ref{StrongHomKer}) becomes
\begin{equation}
K(z_x+a,\theta_x+\varphi\mod2\pi;z_y+a,\theta_y+\varphi\mod2\pi)=e^{-2a}K(z_x,\theta_x;z_y,\theta_y),\quad\forall (a,\varphi)\in\mathbb{R}\times\mathbb{T},\label{KernelProblemCylinder}
\end{equation}
and its explicit general solution is, in accordance with formula (\ref{KProp}),
\begin{equation}
K(z_x,\theta_x;z_y,\theta_y)=e^{-z_x-z_y}F(z_x-z_y,\theta_x-\theta_y+2\pi\mod2\pi),\label{GeneralSolutionCylinder}
\end{equation}
where $F\in\mathcal{F}_H^\lambda=L(G,\mu)$ is arbitrary.

\subsection{Homogeneous integral kernels on $\mathbb{R}^2$}\label{R2HomKerExample}

Consider $M=\mathbb{R}^2\setminus\{0\}$ with polar coordinates $(r,\theta)$ and $G=\mathbb{R}\times\mathbb{T}$ with group coordinates $(a,\varphi)$ acting by dilations and central rotations,
\begin{equation}
(a,\varphi)(r,\theta)=(e^ar,\theta+\varphi\mod2\pi).\label{R2Dilations}
\end{equation}
The measure in question is the Lebesgue measure $d\mu(r,\theta)=rdrd\theta$ (or any normalization thereof), and the expansion coefficient is
\begin{equation}
\lambda_{(a,\varphi)}=\frac{d\mu(e^ar,\theta+\varphi\mod2\pi)}{d\mu(r,\theta)}=e^{2a}.
\end{equation}
The strong homogeneity condition (\ref{StrongHomKer}) in this language becomes
\begin{equation}
K(e^a r_x,\theta_x+\varphi\mod2\pi;e^a r_y,\theta_y+\varphi\mod2\pi)=\frac1{e^{2a}}K(r_x,\theta_x;r_y,\theta_y),\quad\forall (a,\varphi)\in G.\label{KernelProblemEuclid}
\end{equation}
The action of $G$ on $M$ is simply transitive, and the identification $M\simeq G$ is done through the diffeomorphism $(r,\theta)=(e^z,\varphi)$. With this identification the homogeneity condition (\ref{KernelProblemEuclid}) is transferred to the cylinder exactly as the condition (\ref{KernelProblemCylinder}), of which the general solution (\ref{GeneralSolutionCylinder}) can be rewritten in polar coordinates as
\begin{equation}
K(r_x,\theta_x;r_y,\theta_y)=\frac1{r_xr_y}F\left(\frac{r_x}{r_y},\theta_x-\theta_y+2\pi\mod2\pi\right).\label{GeneralSolutionEuclid}
\end{equation}
\begin{remark}\label{Rem:HomogeneousOperatorsProblemEuclid}
Note that the literature on the subject of integral operators with homogeneous kernels on $\mathbb{R}^2$ (or $\mathbb{R}^n$) contains statements that a homogeneous kernel $K(x,y)$ is a function of $|x|$, $|y|$ and the scalar product $\langle x,y\rangle=|x||y|\cos(\theta_x-\theta_y)$. We can see from formula (\ref{GeneralSolutionEuclid}) that this is not correct. Indeed, a direct inspection will show that any periodic function of $\theta_x-\theta_y$ would work, not necessarily a function of $\cos(\theta_x-\theta_y)$ which is contained in the scalar product. Perhaps the discrepancy comes from the confusion between the rotation group $\mathrm{SO}(2)$ and the full orthogonal group $\mathrm{O}(2)$. If we insist that the kernel $K$ be invariant not only under rotations but also reflections then, indeed, one has to restrict to even functions of $\theta_x-\theta_y$, which must be functions of $\cos(\theta_x-\theta_y)$. However, here we do not consider reflections, and therefore, the correct formula is (\ref{GeneralSolutionEuclid}).
\end{remark}

\subsubsection{Hadamard-Bergman convolution operators}

As an example of a class of integral operators with homogeneous kernels on $\mathbb{R}^2$, or more precisely, the unit disk $\mathbb{D}\subset\mathbb{C}=\mathbb{R}^2$, we consider the so-called Hadamard-Bergman convolution operators introduced in \cite{KS-HB-2020} (see also \cite{KA-HB-2021, KS-HB-2021} for some generalizations),
$$
\operatorname{K}f(z)=\int\limits_\mathbb{D}g(w)f(z\overline{ w})d\mu(w),\,\,\, z,w\in\mathbb{D},
$$
where $f,g\in L^1_{\mathrm{loc}}(\mathbb{D},\mu)$ and $\mu$ is the Lebesgue measure normalized such that the measure of $\mathbb{D}$ is 1. Performing a change of variables to $\xi=z\overline{ w}$ we can rewrite the above integral as
$$
\operatorname{K}f(z)=\frac1{|z|^2}\int\limits_{|z|\cdot\mathbb{D}}g\left(\frac{\overline{\xi}}{\overline{ z}}\right)f(\xi)d\mu(\xi).
$$
Set $M=\mathbb{D}\setminus\{0\}$ and introduce the kernel function $K\in L(M\times M,\mu^{\otimes2})$ by
$$
K(z,w)\doteq\begin{cases}
\frac1{|z|^2}g\left(\frac{\overline{ w}}{\overline{ z}}\right),\quad\mbox{if}\quad|w|<|z|,\\
0\quad\mbox{else}.
\end{cases}
$$
Then we can see that
$$
\operatorname{K}f(z)=\int\limits_\mathbb{D}K(z,w)f(w)d\mu(w).
$$
If the cylinder group $G=\mathbb{R}\times\mathbb{T}$ acts on $M\subset\mathbb{R}^2$ as in formula (\ref{R2Dilations}) then in complex variables it becomes
$$
(a,\varphi)z=e^{a+\imath\varphi}z,\quad\forall z\in M,\quad\forall(a,\varphi)\in G.
$$
It can be easily checked that
$$
K(e^{a+\imath\varphi}z,e^{a+\imath\varphi}w)=\frac1{e^{2a}}K(z,w),\quad\forall z,w\in M,\quad\forall(a,\varphi)\in G,
$$
in full accord with formula (\ref{KernelProblemEuclid}). Thus, Hadamard-Bergman convolution operators are integral operators with homogeneous kernels.

\subsection{Homogeneous integral kernels on a disk in $\mathbb{R}^2$ with a radial measure}

Let $M=(0,R)\times\mathbb{T}\subset\mathbb{R}^2$ be an open disk with radius $R\in(0,+\infty]$, with $M=\mathbb{D}\setminus\{0\}$ and $M=\mathbb{R}^2\setminus\{0\}$ being the most important cases. We will work in the polar coordinates $(r,\theta)$ on $M$, such that $r\in(0,R)$ and $\theta\in[0,2\pi)$. Consider a radial measure $\mu$ on $M$ given by
\begin{equation}
d\mu(r,\theta)=\frac{\gamma(r^2)rdrd\theta}{\pi},\label{RadMeasDisk}
\end{equation}
where $\gamma:(0,R^2)\to[0,+\infty)$ is an a.e. positive locally integrable function. We are looking for a group action of $G=\mathbb{R}\times\mathbb{T}$ on $M$ by radial dilations and rotations as
\begin{equation}
(a,\varphi)(r,\theta)=(r_*(r;a),\theta+\varphi\mod2\pi),\quad\forall(r,\theta)\in M,\quad\forall(a,\varphi)\in G.\label{RadMeasDiskAction}
\end{equation}
This is, of course, not the only form of action for a group of dilations that can be considered for the measure $\mu$, but is arguably the most natural one. The dilation condition $\mu((a,\varphi)\cdot.)=\lambda_{(a,\varphi)}\mu(.)$ can be written as
$$
\gamma(r_*^2)r_*dr_*=\lambda_{(a,\varphi)}\gamma(r^2)rdr,
$$
which is an ordinary differential equation for the function $r\mapsto r_*(r;a)$ for every fixed $a\in\mathbb{R}$. Instead of solving this equation we will look for a transformation $M\ni(r,\theta)\mapsto(\rho(r),\theta)\in(0,\infty)\times\mathbb{T}$ such that the measure $\mu$ is transformed into the Lebesgue measure,
$$
\gamma(r^2)rdr=\rho d\rho.
$$
Denote by $\Gamma_C:(0,R^2)\to(0,+\infty)$ the strictly increasing function
$$
\Gamma_C(t)\doteq\int^t\gamma(s)ds+C
$$
for all $C\in\mathbb{R}$ that make $\Gamma_C$ everywhere positive. This gives our desired solution $\rho(r)^2=\Gamma_C(r^2)$. Then $(M,\mu)\simeq\Gamma_C((0,R^2))\subset\mathbb{R}\setminus\{0\}$, so that in coordinates $(\rho,\theta)$ our problem is reduced to Example \ref{R2HomKerExample}. But in $\mathbb{R}\setminus\{0\}$ we already know how the radial dilations act,
$$
(a,\varphi)(\rho,\theta)=(e^a\rho,\theta+\varphi\mod2\pi),\quad\forall(\rho,\theta)\in\mathbb{R}\setminus\{0\},\quad\forall(a,\varphi)\in G,
$$
which brings us to the equations
\begin{equation}
\Gamma_C(r_*^2)=e^{2a}\Gamma_C(r^2),\quad r_*(r;a)=\sqrt{\Gamma_C^{-1}\left(e^{2a}\Gamma_C(r^2)\right)},\quad\forall r\in(0,R),\quad\forall a\in\mathbb{R}.\label{RadMeasDiskr*}
\end{equation}
In this way every admissible choice of the parameter $C$ gives a 1-parameter group of dilations of the measure $\mu$. And for a fixed $C$, we can ask ourselves what the general form of a homogeneous integral kernel is. Since we already know the answer in coordinates $(\rho,\theta)$, namely, formula (\ref{GeneralSolutionEuclid}), all we need to do is to translate it to the original coordinates $(r,\theta)$,
\begin{equation}
K(r_x,\theta_x;r_y,\theta_y)=\frac1{\sqrt{\Gamma_C(r_x^2)\Gamma_C(r_y^2)}}F\left(\sqrt{\frac{\Gamma_C(r_x^2)}{\Gamma_C(r_y^2)}},\theta_x-\theta_y+2\pi\!\!\!\!\mod2\pi\right),\quad\forall(r_x,\theta_x),(r_y,\theta_y)\in M.\label{RadMeasDiskHomKer}
\end{equation}
Below we will apply this procedure to a few particularly interesting cases.

\subsubsection{The Poincar\'e disk $\mathbb{D}\subset\mathbb{C}$}

The Poincar\'e disk model is the unit disk $\mathbb{D}\subset\mathbb{C}$ equipped with the Riemannian metric
\begin{equation}
ds^2=\frac{dzd\bar z}{(1-|z|^2)^2}.
\end{equation}
If we switch to polar coordinates $z=re^{\imath\theta}$ then the corresponding volume form (measure) will be
\begin{equation}
d\mu(r,\theta)=\frac{rdrd\theta}{\pi(1-r^2)^2}.\label{DiskMeasure}
\end{equation}
Let us take $M=\mathbb{D}\setminus\{0\}$. In terms of formula (\ref{RadMeasDisk}) this corresponds to
$$
\gamma(t)=\frac1{(1-t)^2},\quad\Gamma_C(t)=\frac1{1-t}+C,\quad\forall t\in(0,1),\quad C\ge-1.
$$
Formulae (\ref{RadMeasDiskAction}) and (\ref{RadMeasDiskr*}) give as the action of dilations as
$$
(a,\varphi)(r,\theta)=\left(\sqrt{1-\frac1{e^{2a}(\frac1{1-r^2}+C)-C}},\theta+\varphi\mod2\pi\right),\quad\forall(r,\theta)\in M,\quad\forall(a,\varphi)\in G.
$$
In the limit case $C=-1$ this simplifies to
$$
(a,\varphi)(r,\theta)=\left(\frac{e^ar}{\sqrt{1+(e^{2a}-1)r^2}},\theta+\varphi\mod2\pi\right),\quad\forall(r,\theta)\in M,\quad\forall(a,\varphi)\in G.
$$
In complex coordinate $z=re^{\imath\theta}$ this becomes
$$
(a,\varphi)z=\frac{e^{a+\imath\theta}z}{\sqrt{1+(e^{2a}-1)|z|^2}},\quad\forall z\in M,\quad\forall(a,\varphi)\in G.
$$
By formula (\ref{RadMeasDiskHomKer}), the general form of a homogeneous integral kernel is
$$
K(r_z,\theta_z;r_w,\theta_w)=\frac1{\sqrt{\left(\frac1{1-r_z^2}+C\right)\left(\frac1{1-r_w^2}+C\right)}}F\left(\sqrt{\frac{\frac1{1-r_z^2}+C}{\frac1{1-r_w^2}+C}},\theta_z-\theta_w+2\pi\!\!\!\!\mod2\pi\right).
$$
Setting $G(\eta,e^{\imath\theta})=F(\eta,\theta\mod2\pi)$ we can write this kernel in complex variables as
$$
K(z,w)=\frac1{\sqrt{\left(\frac1{1-|z|^2}+C\right)\left(\frac1{1-|w|^2}+C\right)}}G\left(\sqrt{\frac{\frac1{1-|z|^2}+C}{\frac1{1-|w|^2}+C}},\frac{z
\overline{ w}}{|zw|}\right),\quad\forall z,w\in M.
$$
In the limit case $C=-1$ this simplifies to
$$
K(z,w)=\frac{\sqrt{(1-|z|^2)(1-|w|^2)}}{|zw|}G\left(\frac{|z|\sqrt{1-|w|^2}}{|w|\sqrt{1-|z|^2}},\frac{z\overline{ w}}{|zw|}\right),\quad\forall z,w\in M.
$$

\subsubsection{The disk $\mathbb{D}\subset\mathbb{C}$ with weighted Bergman measure}

Let us take $M=\mathbb{D}\setminus\{0\}$ with measure $\mu_\alpha$ depending on a parameter $\alpha\in(-1,+\infty)$ as follows,
$$
d\mu_\alpha(r,\theta)=\frac{(\alpha+1)(1-r^2)^\alpha rdrd\theta}{\pi}.
$$
In terms of formula (\ref{RadMeasDisk}) this corresponds to
$$
\gamma(t)=(\alpha+1)(1-t)^\alpha,\quad\Gamma_C(t)=-(1-t)^{\alpha+1}+C,\quad\forall t\in(0,1),\quad C\ge1.
$$
Formulae (\ref{RadMeasDiskAction}) and (\ref{RadMeasDiskr*}) give us the action of a dilation $(a,\varphi)\in G$ as
$$
(a,\varphi)(r,\theta)=\left(\sqrt{1-\bigl(C-e^{2a}\left[C-(1-r^2)^{\alpha+1}\right]\bigr)^{\frac1{\alpha+1}}},\theta+\varphi\mod2\pi\right),\quad\forall(r,\theta)\in M.
$$
By formula (\ref{RadMeasDiskHomKer}), the general form of a homogeneous integral kernel is
$$
K(r_z,\theta_z;r_w,\theta_w)=\frac1{\sqrt{\left(C-(1-r_z^2)^{\alpha+1}\right)\left(C-(1-r_w^2)^{\alpha+1}\right)}}
$$
$$
\times F\left(\sqrt{\frac{C-(1-r_z^2)^{\alpha+1}}{C-(1-r_w^2)^{\alpha+1}}},\theta_z-\theta_w+2\pi\!\!\!\!\mod2\pi\right),\quad\forall(r_z,\theta_z),(r_w,\theta_w)\in M,
$$
or in complex variables,
$$
K(z,w)=\frac1{\sqrt{\left(C-(1-|z|^2)^{\alpha+1}\right)\left(C-(1-|w|^2)^{\alpha+1}\right)}}
$$
$$
\times G\left(\sqrt{\frac{C-(1-|z|^2)^{\alpha+1}}{C-(1-|w|^2)^{\alpha+1}}},\frac{z\overline{ w}}{|zw|}\right),\quad\forall z,w\in M,
$$
where $G(\eta,e^{\imath\theta})=F(\eta,\theta\mod2\pi)$.

\subsubsection{Lobachevsky space $\mathbb{H}^2$}

Lobachevsky space $\mathbb{H}^2$ can be seen as the positive sheet of the 2-sheet hyperboloid of unit vectors in Minkowski space $\mathbb{R}^{1,2}$. That is,
\begin{equation}
\mathbb{H}^2=\left\{(x,y,z)\in\mathbb{R}^3\,\vline\quad z^2=1+x^2+y^2,\quad z\in[1,\infty)\right\}.
\end{equation}
We introduce hyperbolic polar coordinates $(r,\theta)$ as follows. If $(\eta,\theta)$ are polar coordinates in the plane $(x,y)$, i.e.,
\begin{equation}
\eta=\sqrt{x^2+y^2},\quad x=\eta\cos\theta,\quad y=\eta\sin\theta,
\end{equation}
then we set $z=\cosh r^2$ for $r\in[0,\infty)$. The measure $\mu$ corresponds to the normalized surface area from the induced Riemannian metric,
\begin{equation}
d\mu(z,\theta)=\frac{\sqrt{2z^2-1}dzd\theta}{2\pi},\quad d\mu(r,\theta)=\frac{\sinh r^2\sqrt{2\cosh^2r^2-1}rdrd\theta}{\pi}.
\end{equation}
In terms of formula (\ref{RadMeasDisk}) this corresponds to
$$
\gamma(t)=\sinh t\sqrt{2\cosh^2t-1},
$$
$$
\Gamma_C(t)=\frac{\cosh t\sqrt{2\cosh^2t-1}}2-\frac{\ln\left(\sqrt{2\cosh^2t-1}+\sqrt{2}\cosh t\right)}{2\sqrt2}+C,\quad\forall t\in(0,1),
$$
$$
C\ge\frac{\ln(1+\sqrt2)}{2\sqrt2}-\frac12.
$$
An explicit formula for $\Gamma_C^{-1}$, and thus for the action of dilations, does not seem to be feasible. Yet a direct substitution of $\Gamma_C$ into formula (\ref{RadMeasDiskHomKer}) will yield an explicit form for the homogeneous integral kernel, which we will not write out here.

\subsection{$\mathrm{GL}(n)$-homogeneous integral kernels on $\mathbb{R}^n$}\label{GL2HomKerExample}

As an instructive example of a Case B situation let us consider integral kernels on the Euclidean space homogeneous with respect to all invertible linear transformations. Namely, let $M=G/H=\mathbb{R}^n\setminus\{0\}$, $G=\mathrm{GL}(n)$, $H=\mathrm{Aff}(n-1)\simeq\mathrm{GL}(n-1)\rtimes\mathbb{R}^{n-1}$ and $d\mu(x)=dx$, $n>1$. In this case we have $\lambda_g=|\det g|$ for all $g\in G$. We note that $\lambda|_H\neq\id$, so that we are in Case B. The first important step is to find $\mathcal{X}^\lambda_H$, which is the set of all $x=aH\in G/H$ such that the equations
$$
aha^{-1}=h',\quad\det h'\neq1,\quad h,h'\in H
$$
have a solution. For $n>2$ we find that $\mathcal{X}^\lambda_H=G/H$ and thus $M_*=\emptyset$, which means that there exists no non-trivial homogeneous kernels. For $n=2$ it turns out that
\begin{equation}
\mathcal{X}^\lambda_H=\mathrm{N}(H)H=\left\{(x,y)\in\mathbb{R}^2\,\vline\quad x\neq 0,\quad y=0\right\},\label{XLambdaForm}
\end{equation}
whereas $H\backslash M_*$ is a single point set. This shows that there exists a unique homogeneous integral kernel $K$ up to a scalar factor. If we take $1=F_*\in\wp(\mathfrak{X}(\mathcal{T}))$ then we find
$$
K(x,y)=K(aH,bH)=\frac{F(a^{-1}bH)}{\lambda_a}=\begin{cases}
\frac1{|[x,y]|}\quad\mbox{for}\quad[x,y]\neq0,\\
0\quad\mbox{else},
\end{cases}
$$
where the cross-product $[x,y]=x_1y_2-x_2y_1$. Formally, the corresponding homogeneous integral operator $\operatorname{K}$ should act as
\begin{equation}
\operatorname{K}f(x_1,x_2)=\int\limits_{\mathbb{R}^2}\frac{f(y_1,y_2)dy_1dy_2}{|[x,y]|},\quad\label{ExtremeK}
\end{equation}
but this integral converges for a.e. $x$ only for $f=0$, and thus does not define a sensible integral operator.

On the other hand, if we take only orientation preserving linear transformations, or matrices with a positive determinant, $G=\mathrm{GL}^+(n)$ and $H=\mathrm{Aff}^+(n-1)\simeq\mathrm{GL}^+(n-1)\rtimes\mathbb{R}^{n-1}$, then for $n>2$ we still have $\mathcal{X}^\lambda_H=G/H$ and no non-trivial homogeneous kernels, whereas for $n=2$ formula (\ref{XLambdaForm}) remains valid but $H\backslash M_*$ is a 2-point set. This shows that we have a 2-dimensional space of homogeneous integral kernels $K$. To be more precise, the constant factors can be chosen independently for two possible orientations of the frame $\{x,y\}$, or equivalently, two possible signs of $[x,y]$,
$$
K(x,y)=\begin{cases}
\frac{C_+}{[x,y]}\quad\mbox{for}\quad[x,y]>0,\\
\frac{C_-}{[x,y]}\quad\mbox{for}\quad[x,y]<0,\\
0\quad\mbox{else}.
\end{cases}
$$
If we choose the factors $C_\pm$ to have opposite signs then we will have the same problem as in (\ref{ExtremeK}), i.e., non-integrability. In order that the integral converges at least conditionally, we need to choose $C_+C_->0$. In particular, there is up to a constant a unique homogeneous integral kernel $K$ antisymmetric with respect to $[x,y]$, given by
$$
K(x,y)=\begin{cases}
\frac1{[x,y]}\quad\mbox{for}\quad[x,y]\neq0,\\
0\quad\mbox{else},
\end{cases}
$$
and the corresponding integral operator $\operatorname{K}$ acts as
\begin{equation}\label{eq:UniqueIntOpn=2}
\operatorname{K}f(x_1,x_2)=\int\limits_{\mathbb{R}^2}\frac{f(y_1,y_2)dy_1dy_2}{[x,y]},
\end{equation}
for all $f$ for which the integral makes sense conditionally. Note that if we formally restrict to $x_2=1$ and $y_2=1$ (no integration over $y_2$) then we find the Hilbert transform, up to a constant factor. Thus, we can think of the operator $\operatorname{K}$ as an extension of the Hilbert transform to $\mathbb{R}^2$, which arises naturally as the unique (up to a constant factor) operator with antisymmetric integral kernel homogeneous with respect to all orientation preserving linear transformations.

The operator (\ref{eq:UniqueIntOpn=2}) after a change of variable can be written in the form
\begin{equation}\label{eq:UniqueIntOpn=2.a}
\operatorname{K}f(x_1,x_2)=-\frac{1}{x_1}\int\limits_{\mathbb{R}}\frac{d\xi}{\frac{x_2}{x_1}-\xi}\int\limits_{\mathbb{R}}f(\eta,\xi\eta)d\eta=
-\frac{1}{x_1}\int\limits_{\mathbb{R}}\left(\frac{1}{\frac{x_2}{x_1}-\xi}\frac{1}{\sqrt{1+\xi^2}}\int\limits_{\mathcal{L}_\xi}f(s)dl(s)\right)d\xi,
\end{equation}
where $\mathcal{L}_\xi=\{(t_1,t_2)\in\mathbb{R}^2: t_1\in\mathbb{R}, t_2=\xi t_1\}$ is the line in the plane depending on $\xi\in\mathbb{R}$, and $dl$ is the length element. Therefore, at least formally, the operator $\operatorname{K}$ is a composition up to a multiplier $-\frac{\pi}{x_1}$ of the Hilbert and Radon transforms
\begin{equation}\label{eq:UniqueIntOpn=2.b}
\operatorname{K}f(x_1,x_2)=-\frac{\pi}{x_1} \mathcal{H}\left[\frac{1}{\sqrt{1+\xi^2}}\left(\mathcal{R}f\right)(\mathcal{L}_\xi)\right](\frac{x_2}{x_1})\, .
\end{equation}
As we have already mentioned in the Introduction, the operator (\ref{eq:UniqueIntOpn=2}) is of interest on its own side and we plan to study its properties in detail in another paper. It is obvious that the image of this operator is always a homogeneous function of order $-1$, so it cannot be bounded, say, in Lebesgue $L^p$ spaces. The study of this operator in another type of spaces can be a natural subject for investigation.

However, there is one more general question inspired by the content of this subsection, which we will formulate as an open problem.

\textbf{Open problem:} As noted above, in the case $n>2$ there are no non-trivial kernels homogeneous with respect to all (orientation preserving) linear transformations. However, undoubtedly, the narrowing down of the group of transformations will give nontrivial kernels with certain characteristics. An interesting question therefore is: what would be a natural choice of a subgroup $G\subset\mathrm{GL}^+(n)$ for $n>2$ that produces an essentially unique homogeneous integral kernel $K$ similar to the case $n=2$ above, and what properties does the integral operator with that kernel possess?

\section{Conclusion}
The paper introduces, investigates and systematizes the concept of a homogeneous operator in a general context. Conditions for the homogeneity of the integral operator are given and questions of the weak and strong homogeneity of the kernel of the integral operator are discussed. In a sense, this work aims to provide a clear definition of the above objects and relations between these notions. Operators with homogeneous kernels, which arise in numerous applications in the multidimensional and one-dimensional cases, have been thoroughly studied. The one-dimensional case also goes back to the Hardy-Littlewood-Polya theory. Thus, this subject of research has a long history. However, as far as we know, there has never been an attempt to formalize the concept of a homogeneous operator, and that of a homogeneous integral operator as an important particular case. The general study in this paper is followed by a number of specific examples, each of which may be a candidate for independent study. We expect that this work will give inspirations for new research in the theory of operators and applications, and shed light on the nature of homogeneous operators and, in particular, homogeneous integral operators.

\section{Acknowledgements}
The work was done at the Regional Scientific and Educational Mathematical Center of Southern Federal University with the support of the Ministry of Education and Science of Russia, agreement No. 075-02-2021-1386.
\section{Data availability statement} The authors confirm that data all generated or analysed during this study are included in this article.

\end{document}